\documentclass[12pt,oneside,reqno]{amsart}
\usepackage{graphicx}

\graphicspath{ {figs/} }

\usepackage{pict2e}
\usepackage{amssymb}
\usepackage{amsthm}                % better theorem environments
\usepackage[margin=0.9in]{geometry}  % set the margins to 1in on all sides

\usepackage{epstopdf}
\usepackage{amsmath}
\usepackage{graphicx}
\usepackage{subfigure}
\usepackage{latexsym}
\usepackage{amsmath}
\usepackage{amsfonts}
\usepackage{amssymb}
\usepackage{tikz}
\usepackage{mathdots}
\usepackage{comment}
\usepackage{float}
\usepackage[mathscr]{euscript}
\usepackage{enumerate}
\usepackage{epsfig}
\usepackage { hyperref }
\usepackage { graphics }
\usepackage { graphicx }

\usetikzlibrary{arrows.meta}
\usetikzlibrary{knots}
\usetikzlibrary{hobby}
\usetikzlibrary{arrows,decorations.markings}

\usepackage[utf8]{inputenc}
\usepackage{longtable}
\usepackage[lite]{amsrefs}
\renewcommand{\PrintDOI}[1]{\href{http://dx.doi.org/\detokenize{#1}}{doi: \detokenize{#1}}%
	\IfEmptyBibField{pages}{, (to appear in print)}{}}

\theoremstyle{definition}
\newtheorem{theorem}{Theorem}[section]
\newtheorem{lemma}[theorem]{Lemma}

\theoremstyle{definition}
\newtheorem{definition}[theorem]{Definition}
\newtheorem{example}[theorem]{Example}
\theoremstyle{remark}

\numberwithin{equation}{section}

\numberwithin{equation}{section}
\title{Cocycle Invariants and Oriented Singular Knots}

\author{Jose Ceniceros}
\address{Hamilton College, NY }
\email{jcenicer@hamilton.edu}

\author{Indu R. Churchill}
\address{State University of New York at Oswego, Oswego, NY }
\email{indurasika.churchill@oswego.edu}

\author{Mohamed Elhamdadi}
\address{University of South Florida, Tampa, USA }
\email{emohamed@math.usf.edu}

\author{Mustafa Hajij}
\address{Santa Clara University,Santa Clara, USA}
\email{hajij@scu.edu}

\date{}
\subjclass[2020]{Primary 57K12, 05C38; Secondary 05A15}
\keywords{Quandles, singular Knots and Links}
\dedicatory{}
\begin{document}
\maketitle 
	
\begin{abstract}
We extend the quandle cocycle invariant to oriented singular knots and links using algebraic structures called \emph{oriented singquandles} and assigning weight functions at both regular and singular crossings.  This invariant coincides with the classical cocycle invariant for classical knots but provides extra information about singular knots and links.  The new invariant distinguishes the singular granny knot from the singular square knot. 
%obtained is stronger than the coloring invariant as it distinguishes the singular granny knot from the singular square knot.  Furthermore, we give examples of pairs of singular knots and links with the same number of colorings and distinguish between them using the quandle cocycle invariant.  
\end{abstract}

%\tableofcontents

\section{Introduction}

%this needs to be reworded a little bit :
Quandles are algebraic structures whose axioms model the three Reidemeister moves in knot theory.    Around 1982,  the notion of a quandle was introduced independently by Joyce \cite{Joyce} and Matveev \cite{Matveev} and they used it to construct representations of the braid groups.  Joyce and Matveev associated to each knot a quandle that determines the knot up to isotopy and mirror image.  Since then quandles and racks have been investigated by topologists  in order 
to construct knot and link invariants and their higher analogues (see for 
example \cite{EN} and references therein).  %Quandles have been also investigated in the context of quasigroups \cite{E} and recently from ring theoretical point of views \cite{BPS, EFT}.

Singular knot theory was introduced in 1990 by V. A. Vassiliev in \cite{V} as an extension of classical knot theory.  Since then many classical knot invariants have been extended to singular knots. Braid theory for singular knots was given by Birman in \cite{Birman}.  More precisely, Birman introduced singular braids and conjectured that the monoid of singular braids maps injectively into the group algebra of the braid group.  A proof of this conjecture was given by Paris in \cite{Paris}. Fiedler extended the Kauffman state models of the Jones and Alexander polynomials to the context of singular knots \cite{Fiedler}. %This work is extended in \cite{BEH} to the colored Jones polynomial.  
In \cite{Gemein } Gemein investigated extensions of the Artin representation and the Burau representation to the singular braid monoid and the relations between them. Juyumaya and Lambropoulou constructed a Jones-type invariant for singular links using a Markov trace on a variation of the Hecke algebra \cite{JL}.  Recently, quandles have been used to study singular knots \cites{CEHN, BEHY, NOS}.

The authors of \cite{CJKLS} introduced a cohomology theory for quandles, defined state-sum invariants using quandle cocycles as weights and computed the invariants for some families of classical knots and knotted surfaces.  Since then quandle cocycle invariants were generalized, using quandle modules and non-abelian $2$-cocycles in \cite{CEGS}, extended to the biquandle case in \cite{CES}, and to virtual knots in \cite{CN}.  
In this article, We extend the quandle cocycle invariant to oriented singular knots and links using algebraic structures called \emph{oriented singquandles} and assigning weight functions at both classical and singular crossings.  This invariant coincides with the classical cocycle invariant for classical knots but provides extra information about singular knots and links.  The new invariant distinguishes the singular granny knot from the singular square knot.
 %We introduce the notion of quandle cocycle invariant for oriented singular knots and links using algebraic structures called \emph{oriented singquandles} and assigning weight functions at both regular and singular crossings.  As the main application of the invariant, we show that the invariant distinguishes the singular granny knot from the singular square knot, and thus is \emph{stronger} than the number of colorings invariant.  Our inspiration for the definition of the invariant is found in Carter et al.'s paper \cite{CJKLS}. 
 
 This article is organized as follows.  In Section~\ref{review}, we review the necessary ingredients of quandles and give examples. In
Sections~\ref{OSKQ} and~\ref{alg}, the diagrammatics of the generating set of singular Reidemeister moves are presented and are used to motivate the definition of singquandles.  In Section~\ref{axioms2Cocy}, the conditions on both the weights assigned at regular crossings and singular crossings are derived from the diagrammatics. 
The cocycle invariant is extended to singular knots in Section~\ref{inv} and used in Section~\ref{application} to distinguish singular knots and links with the same number of colorings.
%but distinguished using the quandle cocycle invariant. 

\section{Review of Quandles}\label{review}
In this section we review the basics of quandles.  More details can be found for example in \cites{EN, Joyce, Matveev}.  A set $(X,\ast)$ is called a quandle if the following three identities are satisfied.
\begin{eqnarray*}
& &\mbox{\rm (1) \ }   \mbox{\rm  For all $x \in X$,
$x* x =x$.} \label{axiom1} \\
& & \mbox{\rm (2) \ }\mbox{\rm For all $y,z \in X$, there is a unique $x \in X$ such that 
$ x*y=z$.} \label{axiom2} \\
& &\mbox{\rm (3) \ }  
\mbox{\rm For all $x,y,z \in X$, we have
$ (x*y)*z=(x*z)*(y*z). $} \label{axiom3} 
\end{eqnarray*}

 A {\it quandle homomorphism} between two quandles $(X,*)$ and $(Y,\triangleright)$ is a map $f: X \rightarrow Y$ such that $f(x *y)=f(x) \triangleright f(y) $, where
 $*$ and $\triangleright$ 
 denote respectively the quandle operations of $X$ and $Y$.  If furthermore $f$ is a bijection then it is called a 
{\it quandle isomorphism} between $X$ and $Y$.  The set of quandle automorphisms of a quandle $X$ is a group denoted by Aut($X$).\\

\noindent The following are some typical examples.
\begin{itemize}
\item
Any non-empty set $X$ with the operation $x*y=x$ for all $x,y \in X$ is
a quandle called a  {\it trivial} quandle.
\item
Any group $X=G$ with conjugation $x*y=y^{-1} xy$ is a quandle.

\item
Let $G$ be an abelian group.
For elements  
$x,y \in G$, 
define
$x*y \equiv 2y-x$.
Then $\ast$ defines a quandle
structure on $G$ called \emph{Takasaki} quandle.  In case $G=\mathbb{Z}_n$ (integers mod $n$) the quandle is called {\it dihedral quandle}.
This quandle can be identified with  the
set of reflections of a regular $n$-gon
  with conjugation
as the quandle operation.
\item
Any $\Lambda = (\mathbb{Z }[T^{\pm 1}])$-module $M$
is a quandle with
$x*y=Tx+(1-T)y$, $x,y \in M$, called an {\it  Alexander  quandle}.

\item
A {\it generalized Alexander quandle} is defined  by %also regarded as 
a pair $(G, f)$ where 
$G$ is a  group and $f \in {\rm Aut}(G)$,
and the quandle operation is defined by 
$x*y=f(xy^{-1}) y $.  %The previous example corresponds to when $G$ is abelian group. 
\end{itemize}

The axioms of a quandle correspond respectively to the three Reidemeister moves of types I, II and III (see \cite{EN} for example).  In fact, one of the motivations of defining quandles came from knot diagrammatic.  Given a quandle $(X,*)$, axiom (2) states that {\it right multiplication} by $y \in X$, given by
by ${\mathcal{R}}_y(x) = x*y$ for $x \in X$ is a permutation of $X$.
The subgroup of ${\rm Aut}(X)$, generated by the permutations ${\mathcal{R}}_y$, $y \in X$, is 
called the {\it inner automorphism group} of $X$,  and is 
denoted by ${\rm Inn}(X)$. 
A quandle is {\it connected} if ${\rm Inn}(X)$ acts transitively on $X$.
 The operation $\bar{*}$ on $X$ defined by $x\ \bar{*}\ y= {\mathcal{R}}_y^{-1} (x) $
is a quandle operation.%, and $(X,  \bar{*}) $ is called the {\it dual} quandle of $(X, *)$.

\section{Oriented Singular Knots and Quandles}\label{OSKQ}

Similarly to quandles, singular quandles are algebraic structures that can be derived from the generating set of singular Reidemeister moves given in Figure \ref{Rmoves}. Generating sets of oriented singular Reidemeister moves were studied in \cite{BEHY}. The generating set obtained in \cite{BEHY}  is illustrated in Figure~\ref{generatingset}.

\begin{figure}[h] 
\tiny{
\centering
    \includegraphics[scale=0.6]{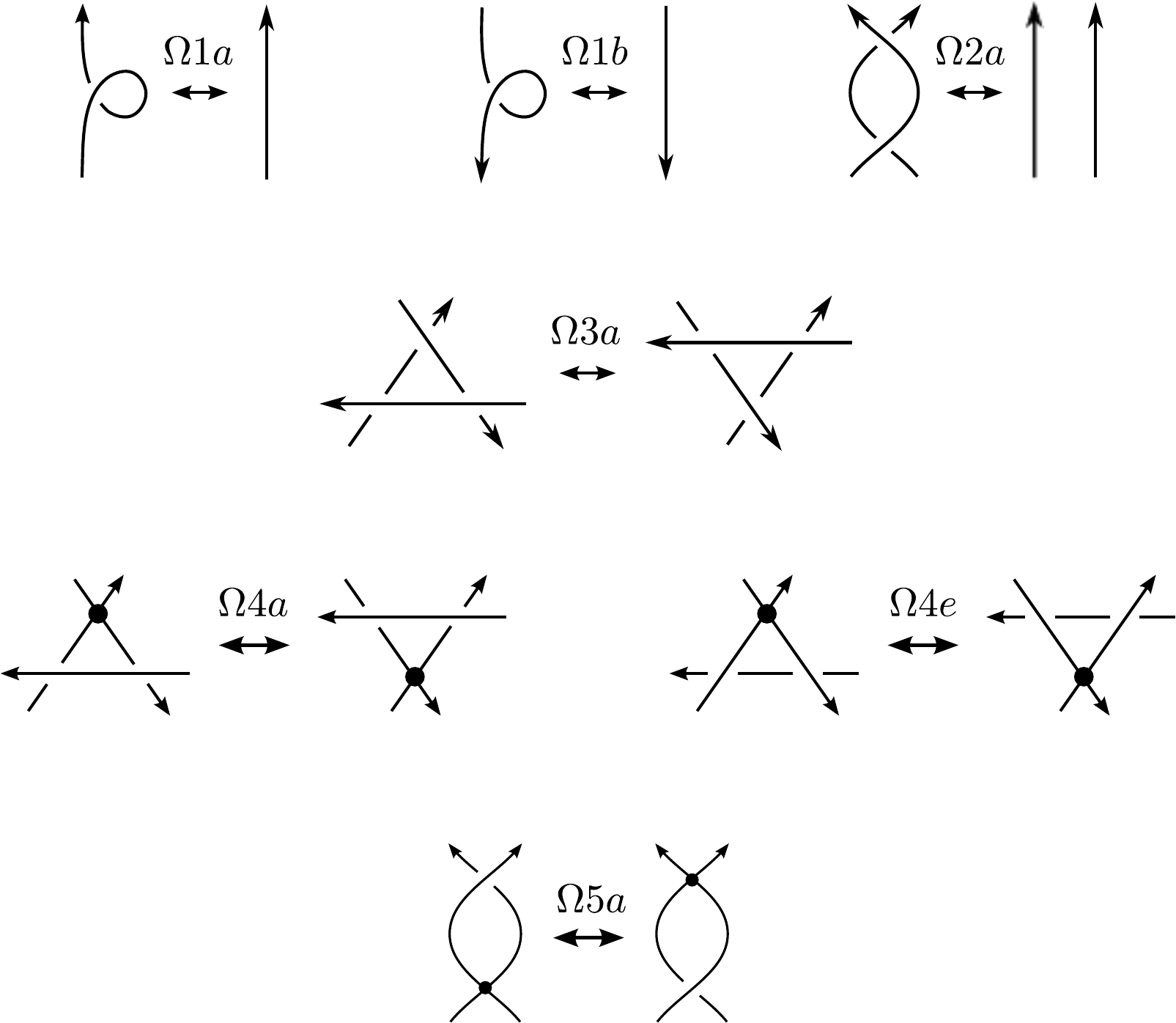}
    \caption{Generating set of singular Reidemeister moves}
    \label{generatingset}}
\end{figure}

Now that we have a generating set of singular Reidemeister moves, the axioms of singquandle can be derived easily.
The first step in this process is to define four binary functions that are indicated in Figure \ref{Rmoves}.

\begin{figure}[h]
	\tiny{
		\centering
		{\includegraphics[scale=0.55]{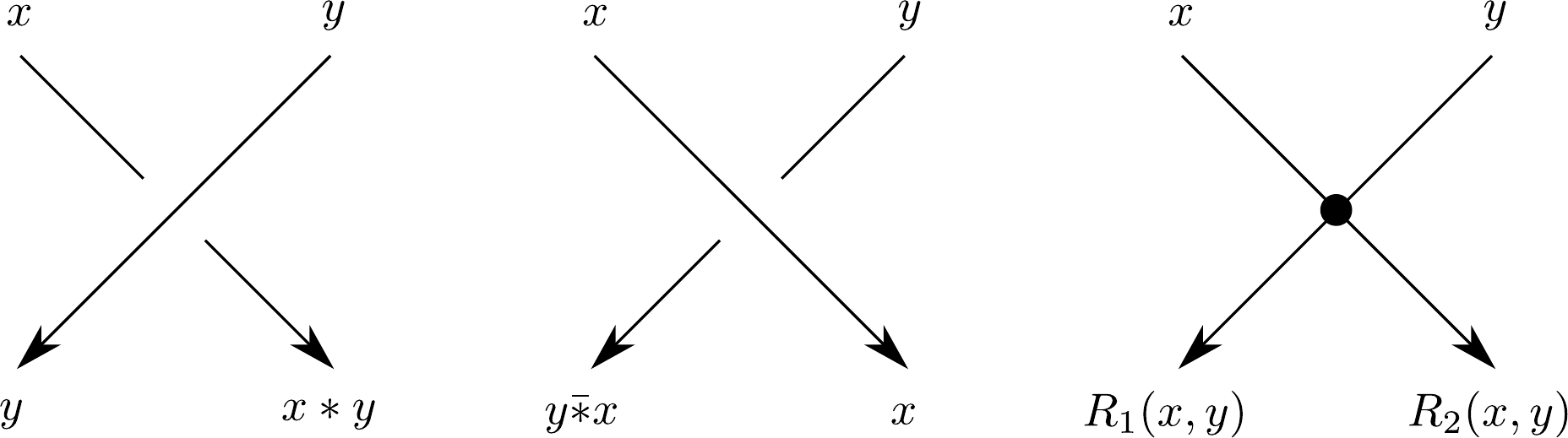}}
		\vspace{.2in}
		\caption{Colorings of classical and singular crossings}
		\label{Rmoves}}
\end{figure}

The axioms of singular quandles are then derived by considering the singular Reidemeister moves and writing the equations obtained using the above binary operations. See Figures   \ref{The generalized Reidemeister move RIV}, \ref{The generalized Reidemeister move RIVb} and \ref{The generalized Reidemeister move RV}.

\begin{figure}[h]
	\tiny{
		\centering
		{\includegraphics[scale=0.55]{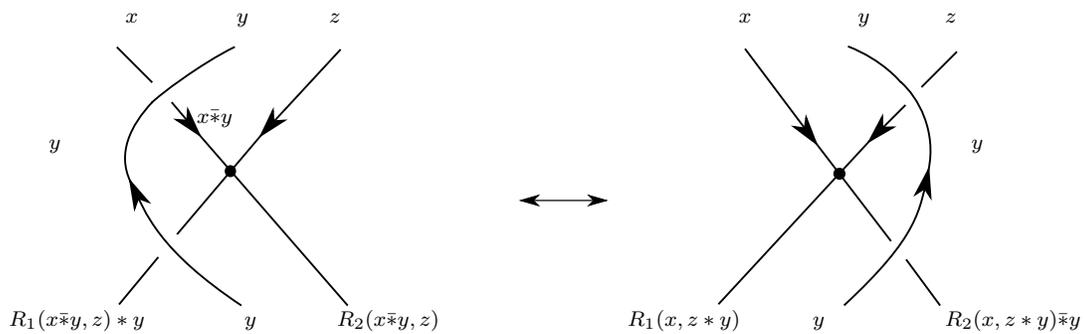}
			\put(-315,108){$x$}
			\put(-273,108){$y$}
			\put(-288,70){$x\bar{*}y$}
			\put(-344,60){$y$}
			\put(5,60){$y$}
			\put(-359,-7){$R_1(x\bar{*}y,z)*y$}
			\put(-270,-7){$y$}
			\put(-235,-7){$R_2(x\bar{*}y,z)$}
			\put(-125,-7){$R_1(x,z*y)$}
			\put(-55,-7){$y$}
			\put(-5,-7){$R_2(x,z*y)\bar{*}y$}
			\put(-238,108){$z$}
			\put(-83,108){$x$}
			\put(-38,108){$y$}
			\put(-5,108){$z$}
		}
		\vspace{.2in}
		\caption{The Reidemeister move $\Omega 4a$ and colorings} 
		\label{The generalized Reidemeister move RIV}}
\end{figure}

\begin{figure}[h]
	\tiny{
		\centering
		{\includegraphics[scale=0.55]{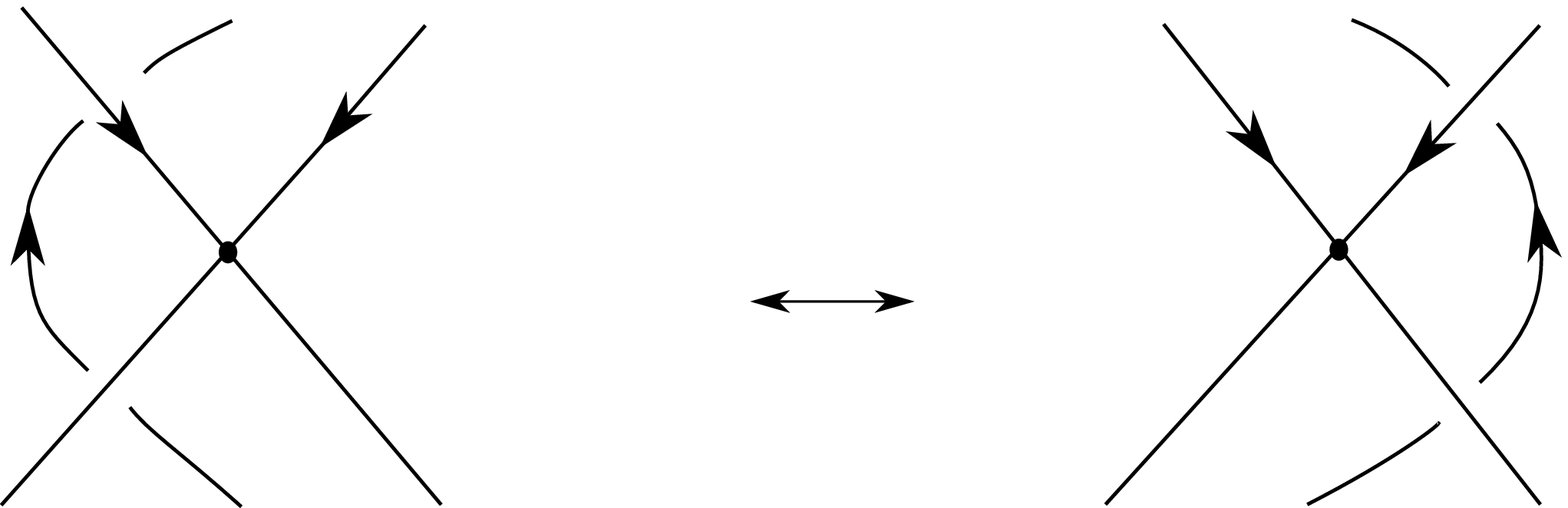}
			\put(-315,108){$x$}
			\put(-300,108){$(y\bar{*}R_1(x,z))*x$}
			\put(-288,70){$x$}
			\put(-364,60){$y\bar{*}R_1(x,z)$}
			\put(5,60){$y*R_2(x,z)$}
			\put(-359,-7){$R_1(x,z)$}
			\put(-280,-7){$y$}
			\put(-235,-7){$R_2(x,z)$}
			\put(-125,-7){$R_1(x,z)$}
			\put(-40,-7){$y$}
			\put(-5,-7){$R_2(x,z)$}
			\put(-228,108){$z$}
			\put(-83,108){$x$}
			\put(-68,108){$(y*R_2(x,z))\bar{*}z$}
			\put(5,108){$z$}
		}
		\vspace{.2in}
		\caption{The Reidemeister move $\Omega 4e$ and colorings }
		\label{The generalized Reidemeister move RIVb}}
\end{figure}

\begin{figure}[h]
\tiny{
  \centering
   {\includegraphics[scale=0.39]{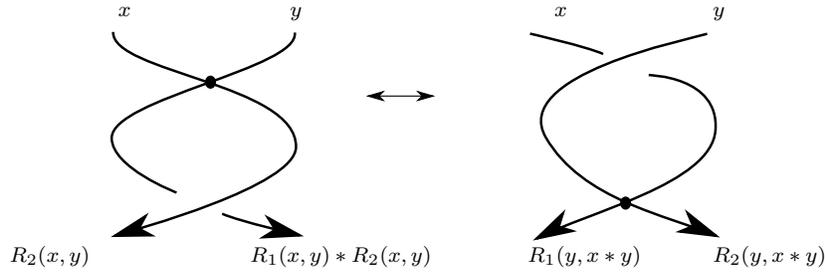}
	\put(-163,85){$y$}
    \put(-228,85){$x$}    
	\put(-3,-8){$R_2(y,x*y)$}
	 \put(-269,-8){$R_2(x,y)$}  
    \put(-63,85){$x$} 
    \put(-3,85){$y$}   
    \put(-73,-8){$R_1(y,x*y)$}
    \put(-178,-8){$R_1(x,y)*R_2(x,y)$}}
     \vspace{.2in}
     \caption{The Reidemeister move $\Omega 5a$ and colorings}
     \label{The generalized Reidemeister move RV}}
\end{figure}

\section{Algebraic Structures from Oriented Singular knots}\label{alg}

The previous three figures immediately gives us the following definition.

\begin{definition}\cite{BEHY} \label{oriented SingQdle}
	Let $(X, *)$ be a quandle.  Let $R_1$ and $R_2$ be two maps from $X \times X$ to $X$.  The triple $(X, *, R_1, R_2)$ is called an { \it oriented singquandle} if the following axioms are satisfied
	\begin{eqnarray}
		R_1(x\bar{*}y,z)*y&=&R_1(x,z*y) \;\;\;\text{coming from $\Omega$4a} \label{eq1}\\
		R_2(x\bar{*}y, z) & =&  R_2(x,z*y)\bar{*}y \;\;\;\text{coming from $\Omega$4a}\label{eq2}\\
	      (y\bar{*}R_1(x,z))*x   &=& (y*R_2(x,z))\bar{*}z \;\;\;\text{coming from $\Omega$4e } \label{eq3}\\
R_2(x,y)&=&R_1(y,x*y)  \;\;\;\text{coming from $\Omega$5a} \label{eq4}\\
R_1(x,y)*R_2(x,y)&=&R_2(y,x*y)  \;\;\;\text{coming from $\Omega$5a} \label{eq5}	
\end{eqnarray}
\end{definition}

The following lemma gives a family of examples of \emph{singquandles} over $\mathbb{Z}_n$.\\

\begin{lemma}
Let $n$ be a positive integer, let $a$ be an invertible element in $\mathbb{Z}_n$ and let $b,c \in \mathbb{Z}_n$, then the binary operations $x*y = ax+(1-a)y$, $ R_1(x,y) = bx + cy$ and $R_2(x,y)= acx + [b+ c(1 - a)]y $ make the triple $(\mathbb{Z}_n,*, R_1,R_2)$ into an oriented singquandle.
\end{lemma}

\begin{proof}
First we have the inverse operation: $x\ \bar{*}\ y = a^{-1}x+(1-a^{-1})y$.  Now let $R_1(x,y) = bx+cy$, then by axiom~(\ref{eq4}) of Definition~\ref{oriented SingQdle}, we have $R_2(x,y) = acx + (c(1-a) + b)y$. Now using axiom~(\ref{eq2}) of Definition~\ref{oriented SingQdle} we obtain that $(1-a)(1 - b-c)=0 \in \mathbb{Z}_n$. Substituting, we find that the following is an oriented singquandle for any invertible $a$, any $b$ and $c$ in $\mathbb{Z}_n$ such that the condition $(1-a)(1 - b-c)=0$ holds in $\mathbb{Z}_n$.  We thus have:
	\begin{eqnarray*}
    	x*y &=& ax + (1-a)y \\
        R_1(x,y) &=& bx + cy \\
        R_2(x,y) &=& acx + [b+ c(1 - a)]y
    \end{eqnarray*}
\end{proof}
%The following proposition, whose proof can be checked by direct computations, gives examples of some families of \emph{singquandles} over groups.  
%\begin{proposition}
%\label{ppp}
%Let $X=G$ be a non-abelian group with the binary operation $x*y=y^{-1}xy$.  Then, for $n \geq 1$, the following families of maps $R_1$ and $R_2$ give pairwise non-isomorphic oriented singquandles structures $(X, *, R_1, R_2)$ on $G$:
%\begin{enumerate}
%\item
%$R_1(x,y)=x(xy^{-1})^{n}$ and $R_2(x,y)=y(x^{-1}y)^n$,
%\item
%$R_1(x,y)=(xy^{-1})^nx$ and %$R_2(x,y)=(x^{-1}y)^ny,$

%\item
%$R_1(x,y)=x(yx^{-1})^{n+1}$ and 
%$R_2(x,y)=x(y^{-1}x)^n.$
%\end{enumerate}
%Furthermore, in each of the cases (1), (2) and (3), different values of $n$ give also non-isomorphic singular quandles.
%\end{proposition}
\section{Deriving the axioms of 2-Cocycles from singular knots}\label{axioms2Cocy}

Let $A$ be abelian group  and let $(X,*)$  be a quandle. A \textit{$2$-cocycle} of $X$ on $A$ is a function $\phi: X \times X \rightarrow A$ that satisfies some conditions coming from Reidemeister moves. For $x,y \in X $ we think of the value $\phi(x,y) $ as a weight associated to a positive crossing and the weight $-\phi(x,y)$ associated to a negative crossing in a knot diagram as shown in Figure \ref{weights}. The weight function $\phi $ must respect the following three conditions.  First, the function $\phi$ satisfies the so called \emph{$2$-cocycle condition} for all $x,y,z \in X$
\begin{eqnarray*}
\phi(x,y) + \phi (x*y,z)=\phi(x,z) + \phi(x*z,y*z). 
\end{eqnarray*}
The above condition is imposed by Reidemeister move III. Moreover, Reidemeister move I imposes the condition $\phi(x,x)=0$ on a $2$-cocycle. Finally, Reidemeister move II imposes that the total Boltzmann weight should be zero; this condition is automatically satisfied due to the weight at a positive and negative crossing canceling each other (see Figure ~\ref{weights}). 
%The function $\phi$ defined earlier corresponds to a crossing from an oriented link as indicated in Figure \ref{weights}. 
For $x,y \in X $ the values of the function $\phi(x,y)$ on a given crossing are usually called the \emph{Boltzmann weights}. 

Now, we will extend this definition to singular knots and links. To this end, we define a function $\phi^{\prime} :X \times X \to A $ that represents the Boltzmann weight at a singular crossing. This is shown in the right figure in Figure \ref{weights}. 

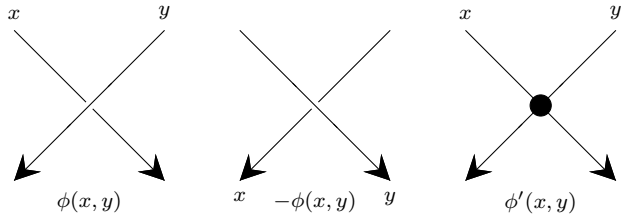
\begin{figure}[h]
\begin{tikzpicture}[use Hobby shortcut]
%diagram on the left
\begin{knot}[
consider self intersections=true,
% draft mode=crossings,
  ignore endpoint intersections=false,
  flip crossing/.list={3,4,5,6,11,13}
]
\strand[decoration={markings,mark=at position 1 with
    {\arrow[scale=3,>=stealth]{>}}},postaction={decorate}] (1,1) ..(-1,-1); 
\strand[decoration={markings,mark=at position 1 with
    {\arrow[scale=3,>=stealth]{>}}},postaction={decorate}] (-1,1) ..(1,-1);
    
\strand[decoration={markings,mark=at position 1 with
    {\arrow[scale=3,>=stealth]{>}}},postaction={decorate}] (2,1) ..(4,-1);
\strand[decoration={markings,mark=at position 1 with
    {\arrow[scale=3,>=stealth]{>}}},postaction={decorate}] (4,1) ..(2,-1);  
    
\strand[decoration={markings,mark=at position 1 with
    {\arrow[scale=3,>=stealth]{>}}},postaction={decorate}] (5,1) ..(7,-1);
\strand[decoration={markings,mark=at position 1 with
    {\arrow[scale=3,>=stealth]{>}}},postaction={decorate}] (7,1) ..(5,-1);      
\end{knot}
\node[above] at (-1,1) {\tiny $x$};
\node[above] at (1,1) {\tiny $y$};
\node[below] at (0,-1) {\tiny $\phi(x,y)$};

\node[below] at (2,-1) {\tiny $x$};
\node[below] at (4,-1) {\tiny $y$};
\node[below] at (3,-1) {\tiny $-\phi(x,y)$};

\node[circle,draw=black, fill=black, inner sep=0pt,minimum size=8pt] (a) at (6,0) {};
\node[above] at (5,1) {\tiny $x$};
\node[above] at (7,1) {\tiny $y$};
\node[below] at (6,-1) {\tiny $\phi'(x,y)$};
\end{tikzpicture}
\vspace{.2in}
		\caption{Weight functions at classical and singular crossings.}
		\label{weights}
\end{figure}

\begin{figure}[h]
	\tiny{
		\centering
		{\includegraphics[scale=0.55]{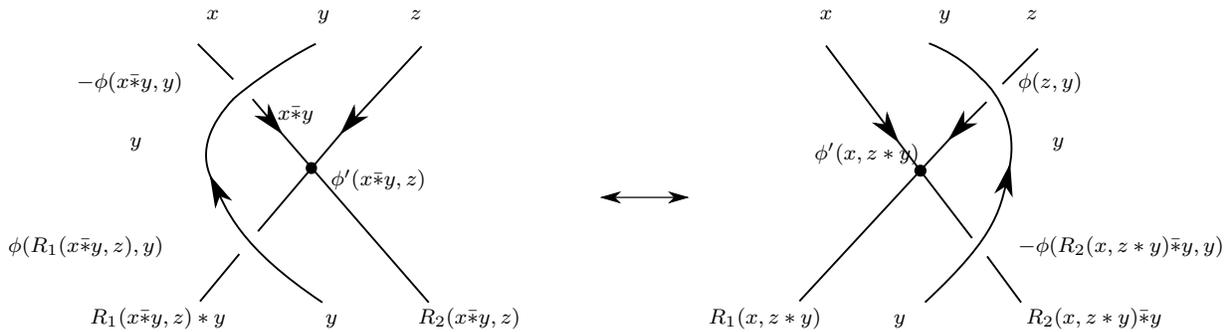}
			\put(-315,108){$x$}
			\put(-273,108){$y$}
			\put(-288,70){$x\bar{*}y$}
			\put(-344,60){$y$}
			\put(5,60){$y$}
			\put(-359,-7){$R_1(x\bar{*}y,z)*y$}
			\put(-270,-7){$y$}
			\put(-235,-7){$R_2(x\bar{*}y,z)$}
			\put(-125,-7){$R_1(x,z*y)$}
			\put(-55,-7){$y$}
			\put(-5,-7){$R_2(x,z*y)\bar{*}y$}
			\put(-238,108){$z$}
			\put(-83,108){$x$}
			\put(-38,108){$y$}
			\put(-5,108){$z$}
            \put(-364,82){$-\phi(x\bar{*}y,y)$}
            \put(-8,82){$\phi(z,y)$}
           \put(-85,55){$\phi^{\prime}(x,z*y)$}
             \put(-268,45){$\phi^{\prime}(x\bar{*}y,z)$}
            \put(-8,20){$-\phi(R_2(x,z*y)\bar{*}y,y)$}
            \put(-390,20){$\phi(R_1(x\bar{*}y,z),y)$}
		}
		\vspace{.2in}
		\caption{The Reidemeister move $\Omega 4a$ and colorings} 
		\label{W4a}}
\end{figure}
Naturally, we now want these two function to satisfy the conditions under singular Reidemeister moves. Using Figure \ref{W4a} we deduce the following condition:
\begin{equation}\label{troisieme}
-\phi(x\bar{*}y,y)+\phi^{\prime}(x\bar{*}y,z)+\phi(R_1(x\bar{*}y,z),y)=\phi(z,y)+\phi^{\prime}(x,z*y)-\phi(R_2(x,z*y)\bar{*}y,y).
\end{equation}

\begin{figure}[h]
	\tiny{
		\centering
		{\includegraphics[scale=0.55]{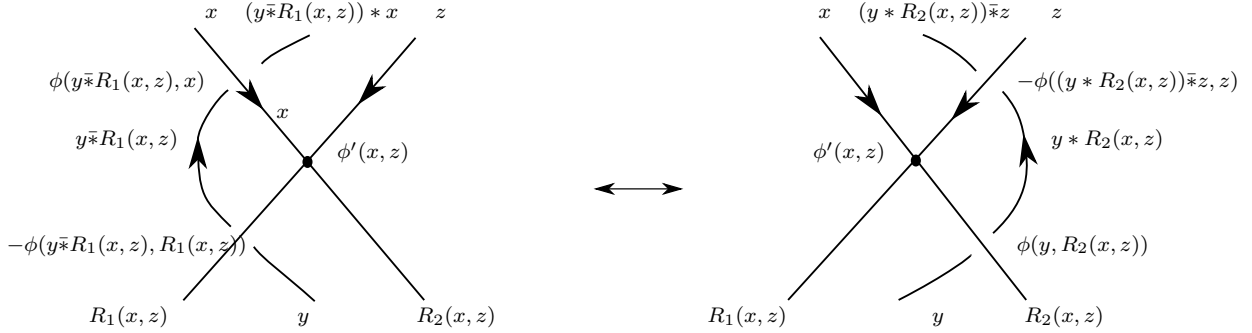}
			\put(-315,108){$x$}
			\put(-300,108){$(y\bar{*}R_1(x,z))*x$}
			\put(-288,70){$x$}
            \put(-374,82){$\phi( y\bar{*}R_1(x,z),x)$}
            \put(-8,82){$-\phi((y*R_2(x,z))\bar{*}z, z)$}
            \put(-85,55){$\phi^{\prime}(x,z)$}
             \put(-265,55){$\phi^{\prime}(x,z)$}
            \put(-8,20){$\phi(y,R_2(x,z))$}
            \put(-390,20){$-\phi(y\bar{*}R_1(x,z),R_1(x,z))$}
			\put(-364,60){$y\bar{*}R_1(x,z)$}
			\put(5,60){$y*R_2(x,z)$}
			\put(-359,-7){$R_1(x,z)$}
			\put(-280,-7){$y$}
			\put(-235,-7){$R_2(x,z)$}
			\put(-125,-7){$R_1(x,z)$}
			\put(-40,-7){$y$}
			\put(-5,-7){$R_2(x,z)$}
			\put(-228,108){$z$}
			\put(-83,108){$x$}
			\put(-68,108){$(y*R_2(x,z))\bar{*}z$}
			\put(5,108){$z$}
		}
		\vspace{.2in}
		\caption{The Reidemeister move $\Omega 4e$ and colorings }
		\label{W4b}}
\end{figure}
Figure \ref{W4b} implies the following equation:
\begin{eqnarray}\label{RR12}
\phi(y\bar{*}R_1(x,z),x)-\phi(y\bar{*}R_1(x,z),R_1(x,z))=
-\phi((y*R_2(x,z))\bar{*}z, z)+\phi(y,R_2(x,z)).
\end{eqnarray}
Notice that in this equation we got rid of the term $\phi'(x,z)$ which appeared on both sides of the equation before simplification.

\begin{figure}[h]
\tiny{
  \centering
   {\includegraphics[scale=0.39]
   {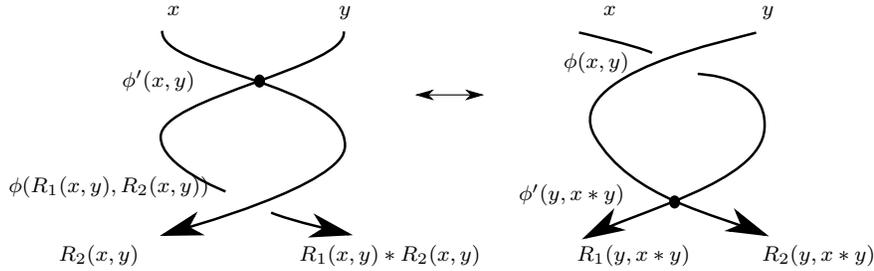}
	\put(-163,85){$y$}
    \put(-228,85){$x$}
    \put(-245,58){$\phi^{\prime}(x,y)$}
    \put(-288,18){$\phi(R_1(x,y),R_2(x,y))$}
    \put(-78,65){$\phi(x,y)$}
    \put(-95,15){$\phi^{\prime}(y,x*y)$}
	\put(-3,-8){$R_2(y,x*y)$}
	 \put(-269,-8){$R_2(x,y)$}  
    \put(-63,85){$x$} 
    \put(-3,85){$y$}   
    \put(-73,-8){$R_1(y,x*y)$}
    \put(-178,-8){$R_1(x,y)*R_2(x,y)$}}
     \vspace{.2in}
     \caption{The Reidemeister move $\Omega 5a$ and colorings}
     \label{W5}}
\end{figure}
Figure \ref{W5} implies that we have the following condition:
\begin{equation}\label{deuxiemeequation}
\phi^{\prime}(x,y)+\phi(R_1(x,y),R_2(x,y))=\phi(x,y)+\phi^{\prime}(y,x*y).
\end{equation}

\section{Cocycle Invariants of Singular knots and links}\label{inv}
In this section we extend the definition of the 2-cocycle invariant to singular knots and prove that it is indeed an invariant for singular knots. First, recall that given an abelian group $A$ and a quandle $(X,*)$, a $2$-cocycle is a function $\phi: X \times X \rightarrow A$ that satisfies the following for all $x,y,z \in X$
\begin{eqnarray}\label{standard}
\phi(x,y) + \phi (x*y,z)&=&\phi(x,z) + \phi(x*z,y*z) \text{ and } \phi(x,x)=0.
\end{eqnarray}
In order to construct our invariant we need to solve the system made of  equations~(\ref{troisieme}),\\ ~(\ref{RR12}), ~(\ref{deuxiemeequation}) and~(\ref{standard}).  First we need to recall the notion of coloring of a knot diagram by an oriented singquandle. Let $D$ be a diagram of a singular knot and let $(X, *, R_1, R_2)$ be an oriented singquandle, then a coloring of $D$ by $(X, *, R_1, R_2)$ is defined in a similar way to the case of colorings of classical knots by quandles.  To be precise, the colorings at positive, negative and singular crossings are given by Figure~\ref{Rmoves}.
%\begin{figure}[h]
%	\tiny{
%		\centering
%		{\includegraphics[scale=0.65]{all_crossings.pdf}}
%		\vspace{.2in}
%		\caption{Colorings at regular and singular crossings}
%		\label{ColoringsCrossings}}
%\end{figure}
As in the case of classical knots we will assign the weights $\phi(x,y)$ and $-\phi(x,y)$ at a positive and negative crossing respectively as shown Figure~\ref{weights}. Furthermore, we associate the weight $\phi'(x,y)$ at a singular crossing as shown in Figure~\ref{weights}.  Now, given an abelian group $A$ (denoted multiplicatively) and cocycles $\phi, \phi': X \times X \rightarrow A$, we define the state sum in exactly the same manner as state sum for classical knots (see \cite{CJKLS}
page 3953).
\begin{definition}\label{StSum}
Let $(X, *, R_1, R_2)$ be an oriented singquandle and let $A$ be an abelian group.  Assume that $\phi, \phi': X \times X \rightarrow A$ satisfy the  equations~(\ref{troisieme}),~(\ref{RR12}), ~(\ref{deuxiemeequation}) and~(\ref{standard}).  Then the state sum for a diagram of a singular knot $K$ is given by $\displaystyle \Phi_{\phi, \phi'}(K)=\sum_{\mathcal{C}}\prod_{\tau} \psi(x,y) $, where the product is taken over all classical and singular crossings $\tau$ of the digaram of $K$, the sum is taken over all possible colorings $\mathcal{C}$ of the knot $K$.   
\end{definition}
Notice that in this definition $\psi(x,y)= 
    \phi(x,y)^{\pm 1}$ at classical positive or negative crossing and $\psi(x,y)= \phi'(x,y)$ at a singular crossing as in Figure~\ref{weights}.
  Now we have the following theorem stating that $\Phi_{\phi, \phi'}(K)$ is an invariant of singular knots.   
\begin{theorem}
Let $\phi, \phi': X \times X \rightarrow A$ be maps satisfying the conditions of Definition~\ref{StSum}.  The \emph{state sum} associated with $\phi$ and $\phi'$ is invariant under the moves listed in the \emph{generating set} of singular Reidemester moves in Figure~\ref{generatingset}, so it defines an invariant of singular knots and links.
\end{theorem}

\begin{proof}
As in classical knot theory, there is one-to-one correspondence between colorings before and after each of the \emph{generating set} of singular Reidemester moves.  The invariance follows automatically from  equations~(\ref{troisieme}),~(\ref{RR12}), ~(\ref{deuxiemeequation}) and~(\ref{standard}) involving $\phi$ and $\phi'$. 
\end{proof}

\section{Distinguishing singular knots and links using the cocycle invariant}\label{application}

In this section, we first provide examples of singquandles and their $2$-cocycles. We then use the cocycle invariant introduced in Definition~\ref{StSum} to distinguish pairs of singular knots and links. Notice that since the target groups of the maps $\phi$ and $\phi'$ will be finite cyclic groups $A=<u, u^n=1>$, then the quandle cocycle invariants will have the form $\displaystyle \Phi_{\phi, \phi'}(K)=\sum_{i=0}^{n-1}a_iu^i$.

The following examples were computed independently using \textit{Mathematica} and \textit{Maple}.  They were also checked by hand computations.\\

\begin{example}\label{ex1cocycles}
Now we give a list of four singquandles and their $2$-cocycles.  Precisely, in each of the following four items $(X,*,R_1,R_2)$ is a singquandle and the maps $\phi$ and $\phi'$ satisfy the equations~(\ref{troisieme}),~(\ref{RR12}), ~(\ref{deuxiemeequation}) and~(\ref{standard}), thus allowing for the computation of the cocycle invariants. 
\begin{enumerate}
    \item 
    Let $X$ be the singquandle $\mathbb{Z}_6$ with the operations $*$ and $\bar{*}$ and the functions $R_1$ and $R_2$ given by, $\forall x,y \in X$, $x *y = -x+2y = x \bar{*}y$ and $R_1(x,y) =3+2x-y$ and $R_2(x,y) = 3+x$. Let the coefficient group be $A=\mathbb{Z}_6$ and let the classical $2$-cocycle be defined by %$\phi(x,y)=(x-y)(2-x+2y)$ 
    $\phi(x,y) = 2x+3x^2 -2y-xy -2y^2$ and the weight at singular crossings be defined by $\phi'(x,y) = 3+x+x^2+2y-xy$.   
    
    \item
    Let $X$ be the singquandle $\mathbb{Z}_6$ with the operations $*$ and $\bar{*}$ and the functions $R_1$ and $R_2$ given by, $\forall x,y \in X$, $x *y = -x +2y = x 
\bar{*}y$ and $R_1(x,y) =3+x$ and $R_2(x,y) = 3+3x+3x^2+y$. Let the coefficient group be $A=\mathbb{Z}_2$ and let the classical $2$-cocycle be 
defined by $\phi(x,y) = y(x+1)$ and the weight at singular crossings be defined by $\phi'(x,y) = 1+x+xy$.

\item
Let $X$ be the singquandle $\mathbb{Z}_{10}$ with the operations $*$ and $\bar{*}$ and the functions $R_1$ and $R_2$ given by, $\forall x,y \in X$, $x *y = 3x -2y$, $x 
\bar{*}y = -3x+4y$ with $R_1(x,y) =x$ and $R_2(x,y) = 5x+5x^2+y$. Let the coefficient group be $A=\mathbb{Z}_4$ and let the classical $2$-cocycle be 
define by $\phi(x,y) = (x-y)(y-x)$ and the weight at singular crossings be defined by $\phi'(x,y) = 1+2x+3x^2+y^2$.

%\item
%Let $X$ be the singquandle $\mathbb{Z}_{10}$ with the operations $*$ and $\bar{*}$ and the functions $R_1$ and $R_2$ given by, $\forall x,y \in X$, $x*y = 3x -2y$, $x 
%\bar{*}y = -3x+4y$ with $R_1(x,y) =3(x+y)$ and $R_2(x,y) = 4x+5x^2-3y$. Let the coefficient group be $A=\mathbb{Z}_4$ and let the classical $2$-cocycle be 
%define by $\phi(x,y) = 2x(x+1)\equiv 0$ (since $x(x+1)$is always even), and the weight at singular crossings be defined by $\phi'(x,y) = -1+2x+x^2-y^2$.

%correction
\item
Let $X$ be the singquandle $\mathbb{Z}_{10}$ with the operations $*$ and $\bar{*}$ and the functions $R_1$ and $R_2$ given by, $\forall x,y \in X$, $x*y = 3x -2y$, $x 
\bar{*}y = -3x+4y$ with $R_1(x,y) =3x+8y$ and $R_2(x,y) = 4x+7y$. Let the coefficient group be $A=\mathbb{Z}_4$ and let the classical $2$-cocycle be 
define by $\phi(x,y) = 2(2+3x+3y)^3$, and the weight at singular crossings be defined by $\phi'(x,y) = (1+x+y)^5$.
\end{enumerate}
\end{example}

Now we give examples of pairs of singular knots and links and we distinguish them using the list of singquandles and their $2$-cocycles given in Example~\ref{ex1cocycles}.

\begin{example}
In this example we show that the cocycle invariant distinguishes two singular knots with 5 classical crossings and a singular crossing. Consider the oriented singquandle $(\mathbb{Z}_{6}, *, R_1, R_2)$ with cocycles $\phi(x,y) = 2x+3x^2 -2y-xy -2y^2$ and $\phi'(x,y) = 3+x+x^2+2y-xy$ from item (1) of Example \ref{ex1cocycles}. We will use the $2$-cocycle invariant to distinguish the following singular knots listed as $5^k_1$ and $5^k_8$ in \cite{Oyamaguchi}. 
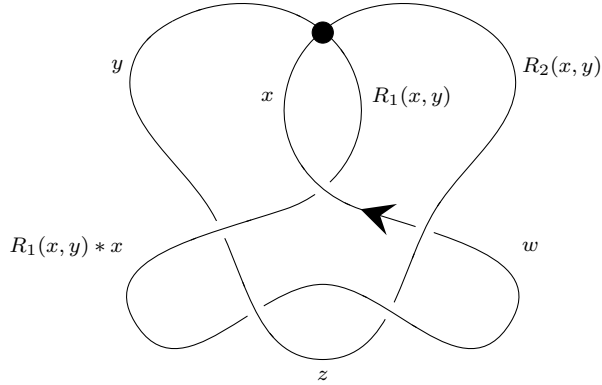
\begin{figure}[h]
\begin{tikzpicture}[use Hobby shortcut]
%diagram on the left
\begin{knot}[thick,
consider self intersections=true,
%  draft mode=crossings,
  ignore endpoint intersections=false,
  flip crossing/.list={2,4,6,8}
]
\strand ([closed]-2,2)[decoration={markings,mark=at position .5 with
    {\arrow[scale=3,>=stealth]{>}}},postaction={decorate}]..(.5,1)..(-.5,-.5)..(-2.5,-2)..(0,-1.5)..(2.5,-2)..(.5,-.5)..(-.5,1)..(2,2)..(2.5,1.5)..(1.5,-.5)..(0,-2.5)..(-1.5,-.5)..(-2.5,1.5);
\end{knot}

\node[circle,draw=black, fill=black, inner sep=0pt,minimum size=8pt] (a) at (0,1.85) {};
\node[right] at (.5,1) {\tiny $R_1(x,y)$};
\node[right] at (2.5,1.4) {\tiny $R_2(x,y)$};
\node[left] at (-.5,1) {\tiny $x$};
\node[left] at (-2.5,1.4) {\tiny $y$};
\node[right] at (2.5,-1) {\tiny $w$};
\node[left] at (-2.5,-1) {\tiny $R_1(x,y)*x$};
\node[below] at (0,-2.5) {\tiny $z$};
\end{tikzpicture}
\vspace{.2in}
		\caption{Diagram for $5^k_1$.}
		\label{5k1}
\end{figure}

\begin{figure}[h]
\begin{tikzpicture}[use Hobby shortcut]
%diagram on the left
\begin{knot}[thick,
consider self intersections=true,
% draft mode=crossings,
  ignore endpoint intersections=false,
  flip crossing/.list={7,8,4,10}
]
\strand ([closed]0,2)[decoration={markings,mark=at position .5 with
    {\arrow[scale=3,>=stealth]{>}}},postaction={decorate}]..(-.5,1.75)..(0,0)..(.5,-.2)..(2,-2)..(-.5,-1)..(-2,1.5)..(0,1.5)..(2,1.5)..(.5,-1)..(-2,-2)..(-.5,-.2)..(0,0)..(.5,1.75)..(0,2);
\end{knot}
\node[circle,draw=black, fill=black, inner sep=0pt,minimum size=8pt] (a) at (0,0) {};
\node[right] at (.35,0) {\tiny $R_1(x,y)$};
\node[right] at (.7,.9) {\tiny $R_2(x,y)$};
\node[left] at (-.5,0) {\tiny $x$};
\node[left] at (-.7,1) {\tiny $y$};
\node[left] at (-2.4,1) {\tiny $w$};
\node[right] at (2.1,-1.2) {\tiny $R_1(x,y)*z$};
\node[right] at (2.4,1) {\tiny $z$};
\end{tikzpicture}
\vspace{.2in}
		\caption{Diagram for $5^k_8$.}
		\label{5k8}
\end{figure}
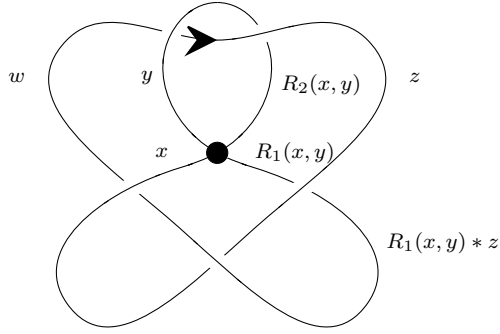

The coloring of $5^k_1$ by the above signquandle gives the set of equations

\[
\begin{cases}
                         x=w\bar{*}R_2(x,y) = -w+2x\\
                          y=z\bar{*}(R_1(x,y)*x) = 2y-z\\
                     z=R_2(x,y) \bar{*} w = 3 + 2w-x\\
                          w=(R_1(x,y)*x)\bar{*}z=3-y+2z.
\end{cases}
\]

Coloring the singular knot $5^k_8$ by the above singquandle gives the set of equations
\[
\begin{cases}
                                       x=z*(R_1(x,y)*z)=2x+2y+3z\\
                          y=R_2(x,y) * z = 3-x + 2z\\
                    z=w*y = -w+2y\\
                          w=(R_1(x,y)*z)*x = 3+4x-y+4z.
\end{cases}
\]
These two systems give that the two singular knots have 6 colorings by this singquandle.  Now we compute the $2$-cocycle invariant using $\phi,\phi' : X \times X \rightarrow \mathbb{Z}_6$ as defined above. The singular knot $5^k_1$ has $\Phi_{\phi,\phi'}(5^k_1) %= \sum u^{\phi(x,y) + \phi'(x,y)}
= 6u^3$ and the second knot $5^k_8$ has $\Phi_{\phi,\phi'}(5^k_8) 
%= \sum u^{\phi(x,y) + \phi'(x,y)} 
= 6u^0= 6$, thus distinct.
\end{example}

\begin{example}
In this example we show that the cocycle invariant detects the change in the singular crossing. Consider the oriented singquandle $(\mathbb{Z}_{6}, *, R_1, R_2)$ with cocycles $\phi(x,y) = 2x+3x^2 -2y-xy -2y^2$ and $\phi'(x,y) = 3+x+x^2+2y-xy$ from item (1) of Example \ref{ex1cocycles}. We will use the $2$-cocycle invariant to distinguish the following singular knots listed as $5^k_6$ and $5^k_7$ in \cite{Oyamaguchi}. 

\begin{figure}[h]
\begin{tikzpicture}[use Hobby shortcut]
%diagram on the left
\begin{knot}[thick,
consider self intersections=true,
% draft mode=crossings,
  ignore endpoint intersections=false,
  flip crossing/.list={3,4,5,6,11,13}
]
\strand ([closed]0,2) [decoration={markings,mark=at position .5 with
    {\arrow[scale=3,>=stealth]{>}}},postaction={decorate}]..(-1,1.5)..(0,1)..(.5,1.5)..(2,0)..(0,-1.5)..(-2,0)[decoration={markings,mark=at position .93 with
    {\arrow[scale=3,>=stealth]{>}}},postaction={decorate}]..(-1,1.5)..(0,1)..(0,0)..(1,-1)..(0,-2)..(-1,-1)..(0,-.3)..(1,1)..(.4,1.8)..(0,2);
\end{knot}

\node[circle,draw=black, fill=black, inner sep=0pt,minimum size=8pt] (a) at (-1,1.5) {};
\node[above] at (-1.6,1.4) {\tiny $x$};
\node[above] at (0,2) {\tiny $y$};
\node[below] at (0,-2) {\tiny $w$};
\node[right] at (.7,-.4) {\tiny $z$};
\end{tikzpicture}
\vspace{.2in}
		\caption{Diagram for $5^k_6$.}
		\label{5k6}
\end{figure}
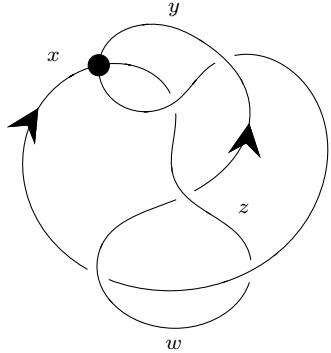

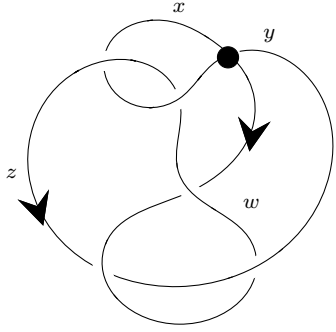
\begin{figure}[h]

\begin{tikzpicture}[use Hobby shortcut]
%diagram on the left
\begin{knot}[thick,
consider self intersections=true,
% draft mode=crossings,
  ignore endpoint intersections=false,
  flip crossing/.list={3,4,5,6,11,13}
]
\strand ([closed]0,2)[decoration={markings,mark=at position .45 with
    {\arrow[scale=3,>=stealth]{<}}},postaction={decorate}]..(-1,1.5)..(0,1)..(.5,1.5)..(2,0)..(0,-1.5)..(-2,0)[decoration={markings,mark=at position .93 with
    {\arrow[scale=3,>=stealth]{<}}},postaction={decorate}]..(-1,1.5)..(0,1)..(0,0)..(1,-1)..(0,-2)..(-1,-1)..(0,-.3)..(1,1)..(.4,1.8)..(0,2);
\end{knot}

\node[circle,draw=black, fill=black, inner sep=0pt,minimum size=8pt] (a) at (.65,1.55) {};
\node[above] at (0,2) {\tiny $x$};
\node[above] at (1.2,1.6) {\tiny $y$};
\node[right] at (.7,-.4) {\tiny $w$};
\node[left] at (-2,0) {\tiny $z$};
\end{tikzpicture}
\vspace{.2in}
		\caption{Diagram for $5^k_7$.}
		\label{5k7}
\end{figure}

Coloring of $5^k_6$ by the above signquandle gives the set of equations
\[
\begin{cases}
                         x=(R_1(x,y) \bar{*}y)*w = 3+2w + 2x+3y\\
                          y=w*z= -w+2z\\
                     z=R_2(x,y)*R_1(x,y)=3+3x-2y\\
                          w=z*(R_1(x,y)\bar{*}y)= 2x-z.
\end{cases}
\]

Coloring the singular knot $5^k_7$ by the above singquandle gives the set of equations
\[
\begin{cases}
                                       x=R_1(x,y)\bar{*}z= 3-2x+y+2z\\
                          y=z*(R_2(x,y)*w)= 4w-2x-z\\
                    z=w*R_1(x,y)=-w+4x-2y\\
                          w=(R_2(x,y)*w)*y=3-2w+x+2y.
\end{cases}
\]
These two systems give that the two singular knots have 6 colorings by this singuqndle.  Now we compute the $2$-cocycle invariant using $\phi,\phi': X \times X \rightarrow \mathbb{Z}_6$ as defined above. The singular knot $5^k_6$ has $\Phi_{\phi,\phi'}(5^k_6) 
%= \sum u^{\phi(x,y) + \phi'(x,y)}
= 6u^3$ and the second knot $5^k_7$ has $\Phi_{\phi,\phi'}(5^k_7) 
%= \sum u^{\phi(x,y) + \phi'(x,y)} 
= 6u^0= 6$, thus distinct.
\end{example}

\begin{example}  Consider the oriented singquandle $(\mathbb{Z}_{6}, *, R_1, R_2)$ with cocycles $\phi(x,y)= y(x+1)$ and $\phi'(x,y) = 1 + x + xy$ from item (2) of Example \ref{ex1cocycles}.  Now we color the two singular links $K_1$ and $K_2$ by the oriented singquandle $(\mathbb{Z}_{6}, *, R_1, R_2)$ as shown in the figure below.  
\begin{figure}[h]
\begin{tikzpicture}[use Hobby shortcut]
%diagram on the left
\begin{knot}[thick,
%  draft mode=crossings,
  flip crossing=4,
]
\strand[decoration={markings,mark=at position .5 with
    {\arrow[scale=3,>=stealth]{>}}},postaction={decorate}] (2.5,0) circle[radius=2cm];
\strand[decoration={markings,mark=at position .5 with
    {\arrow[scale=3,>=stealth]{>}}},postaction={decorate}] (0,0) circle[radius=2cm];
\strand[decoration={markings,mark=at position .5 with
    {\arrow[scale=3,>=stealth]{>}}},postaction={decorate}] (-2.5,0) circle[radius=2cm];
\end{knot}
\node[circle,draw=black, fill=black, inner sep=0pt,minimum size=8pt] (a) at (1.24,1.55) {};
\node[circle,draw=black, fill=black, inner sep=0pt,minimum size=8pt] (a) at (-1.24,1.55) {};
\node[above] at (0,2) {\tiny $R_1(x,y)$};
\node[left] at (.9,1.4) {\tiny $R_2(x,y)$};
\node[left] at (2,2.1) {\tiny $x$};
\node[right] at (2,1) {\tiny $y$};
\node[left] at (-4.4,1) {\tiny $w$};
\node[left] at (-1.6,1.3) {\tiny $R_2(R_1(x,y),w)$};
\node[below] at (0,-2) {\tiny $z$};
\end{tikzpicture}
\vspace{.2in}
		\caption{Diagram $K_1$.}
		\label{consumK1}
\end{figure}
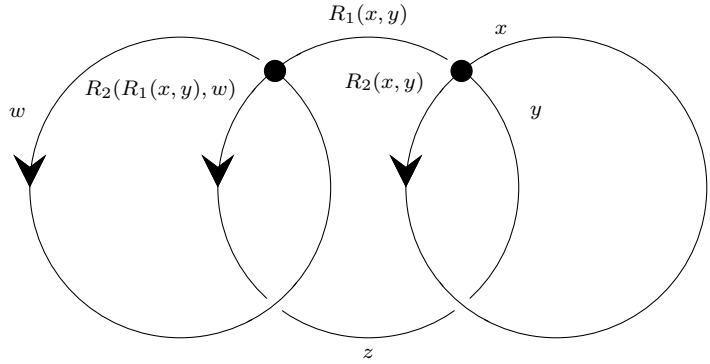

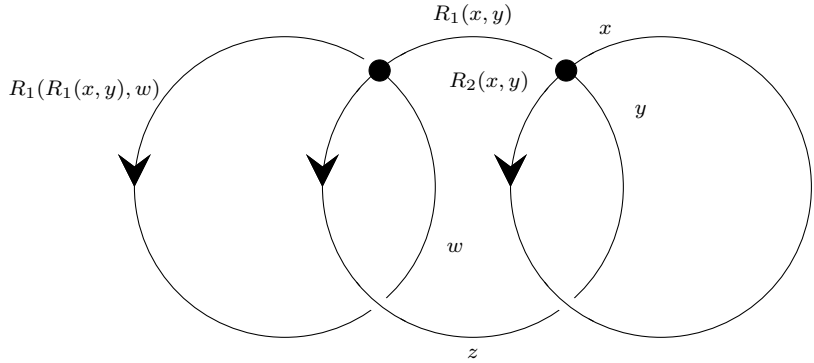
\begin{figure}[h]
\begin{tikzpicture}[use Hobby shortcut]
%diagram on the left
\begin{knot}[thick,
%  draft mode=crossings,
 % flip crossing=4
]

\strand[decoration={markings,mark=at position .5 with
    {\arrow[scale=3,>=stealth]{>}}},postaction={decorate}] (2.5,0) circle[radius=2cm];
\strand[decoration={markings,mark=at position .5 with
    {\arrow[scale=3,>=stealth]{>}}},postaction={decorate}] (0,0) circle[radius=2cm];
\strand[decoration={markings,mark=at position .5 with
    {\arrow[scale=3,>=stealth]{>}}},postaction={decorate}](-2.5,0) circle[radius=2cm];

\end{knot}
\node[circle,draw=black, fill=black, inner sep=0pt,minimum size=8pt] (a) at (1.24,1.55) {};
\node[circle,draw=black, fill=black, inner sep=0pt,minimum size=8pt] (a) at (-1.24,1.55) {};
\node[above] at (0,2) {\tiny $R_1(x,y)$};
\node[left] at (.9,1.4) {\tiny $R_2(x,y)$};
\node[left] at (2,2.1) {\tiny $x$};
\node[right] at (2,1) {\tiny $y$};
\node[right] at (-0.5,-.8) {\tiny $w$};
\node[below] at (0,-2) {\tiny $z$};
\node[left] at (-4,1.3) {\tiny $R_1(R_1(x,y),w)$};
\end{tikzpicture}
\vspace{.2in}
		\caption{Diagram $K_2$.}
		\label{consumK2}
\end{figure}
The coloring of the first singular link $K_1$ by this singquandle gives the set of equations 
\[
\begin{cases}
                         x=R_2(x,y)=3+3x+3x^2+y\\
                          y=z*R_2(x,y)=2y-z\\
                     z=R_2(R_1(x,y),w)\bar{*}w=3+w+3x+3x^2\\
                          w=R_1(R_1(x,y),w)=x.
\end{cases}
\]

Coloring the second singular link $K_2$ by this singquandle gives the set of equations 
\[
\begin{cases}
                                       x=R_2(x,y)=3+3x+3x^2+y\\
                          y=z*R_2(x,y)=2y+5z\\
                    z=R_2(R_1(x,y),w)=3+w+3x+3x^2\\
                          w=R_1(R_1(x,y),w)*z=-x+2z.
\end{cases}
\]
These two systems give that the two singular links $K_1$ and $K_2$ have 6 colorings by this singquandle.  We now compute the $2$-cocycle invariant using $\phi,\phi': X \times X \rightarrow \mathbb{Z}_2$ as defined above. The first singular link $K_1$ has $\Phi_{\phi,\phi'}(K_1) 
%= \sum u^{\phi(x,y) + \phi'(x,y)}
= 6u$ and the second link $K_2$ has $\Phi_{\phi,\phi'}(K_2) 
%= \sum u^{\phi(x,y) + \phi'(x,y)} 
= 6u^0= 6$.  Thus the two links are distinct.
\end{example}

\begin{example}  Consider the oriented singquandle $(\mathbb{Z}_{10}, *, R_1, R_2)$ with cocycles $\phi(x,y) = (x-y)(y-x)$ and $\phi'(x,y) = 1+2x+3x^2+y^2$ from item (3) of Example \ref{ex1cocycles}.  Now we color the two singular knots $K_1$ and $K_2$ by the oriented singquandle $(\mathbb{Z}_{10}, *, R_1, R_2)$ as shown in the figure below.  

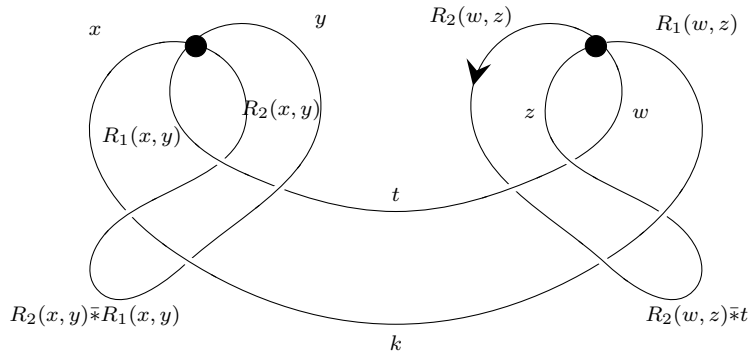
\begin{figure}[h]

\begin{tikzpicture}[use Hobby shortcut]
%diagram on the left
\begin{knot}[
consider self intersections=true,
%draft mode=crossings,
  ignore endpoint intersections=false,
  flip crossing/.list={4,7,9,12}
]
\strand ([closed]-4,1)[decoration={markings,mark=at position .52 with
    {\arrow[scale=3,>=stealth]{>}}},postaction={decorate}]..(-2,1)..(-4,-1)..(-4,-1.5)..(-2.8,-1.2)..(-1,1)..(-2,2)..(-3,1)..(-2,0)..(0,-.5)..(2,0)..(3,1)..(2,2)..(1,1).. (2.8,-1.2)..(4,-1.5)..(4,-1)..(2,1)..(4,1)..(3,-1).. (0,-2)..(-3,-1)..(-4,1);
\end{knot}
\node[circle,draw=black, fill=black, inner sep=0pt,minimum size=8pt] (a) at (-2.659,1.7) {};
\node[circle,draw=black, fill=black, inner sep=0pt,minimum size=8pt] (a) at (2.659,1.7) {};
\node[above] at (-4,1.7) {\tiny $x$};
\node[above] at (-1,1.8) {\tiny $y$};
\node[left] at (-2.7,.5) {\tiny $R_1(x,y)$};
\node[right] at (-2.2,.85) {\tiny $R_2(x,y)$};
\node[above] at (0,-.5) {\tiny $t$};
\node[below] at (0,-2) {\tiny $k$};
\node[right] at (3,.8) {\tiny $w$};
\node[left] at (2,.8) {\tiny $z$};
\node[above] at (4,1.7) {\tiny $R_1(w,z)$};
\node[above] at (1,1.8) {\tiny $R_2(w,z)$};
\node[below] at (-4,-1.6) {\tiny $R_2(x,y) \bar{*} R_1(x,y)$};
\node[below] at (4,-1.6) {\tiny $R_2(w,z)\bar{*}t$};
\end{tikzpicture}
\vspace{.2in}
		\caption{Diagram $K_3$.}
		\label{consumK11}
\end{figure}

The coloring the first singular knot $K_3$ by this singquandle gives the set of equations 
\[
\begin{cases}
                        x=k \bar{*} (R_2(x,y) \bar{*} R_1(x,y)) = 7k+6x+8y\\
                        y=(R_2(x,y) \bar{*} R_1(x,y))\bar{*} k = 4k+3x+5x^2+9y\\
                        t= R_1(x,y) \bar{*}y = 7x+4y\\
                        z = (R_2(w,z) \bar{*} t)\bar{*} R_1(w,z) = 8t+9w+5w^2+9z\\
                        w= t \bar{*}z=7t+4z\\
                        k = R_1(w,z) \bar{*} (R_2(w,z) \bar{*} t)= 6t+7w+8z.
\end{cases}
\]

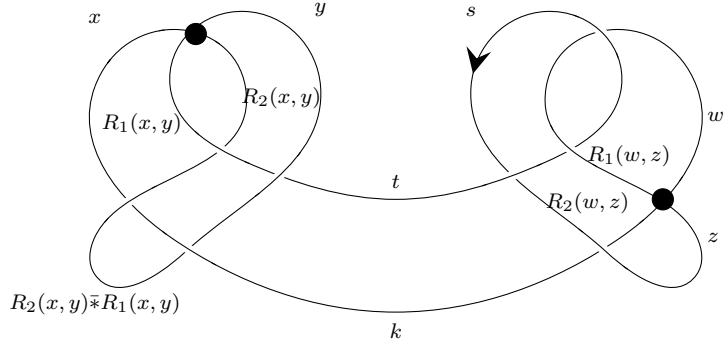
\begin{figure}[h]

\begin{tikzpicture}[use Hobby shortcut]
%diagram on the left
\begin{knot}[
consider self intersections=true,
%draft mode=crossings,
  ignore endpoint intersections=false,
  flip crossing/.list={4,7,9,12}
]
\strand ([closed]-4,1)[decoration={markings,mark=at position .52 with
    {\arrow[scale=3,>=stealth]{>}}},postaction={decorate}]..(-2,1)..(-4,-1)..(-4,-1.5)..(-2.8,-1.2)..(-1,1)..(-2,2)..(-3,1)..(-2,0)..(0,-.5)..(2,0)..(3,1)..(2,2)..(1,1).. (2.8,-1.2)..(4,-1.5)..(4,-1)..(2,1)..(4,1)..(3,-1).. (0,-2)..(-3,-1)..(-4,1);
\end{knot}
\node[circle,draw=black, fill=black, inner sep=0pt,minimum size=8pt] (a) at (-2.659,1.7) {};
\node[circle,draw=black, fill=black, inner sep=0pt,minimum size=8pt] (a) at (3.55,-.5) {};
\node[above] at (-4,1.7) {\tiny $x$};
\node[above] at (-1,1.8) {\tiny $y$};
\node[left] at (-2.7,.5) {\tiny $R_1(x,y)$};
\node[right] at (-2.2,.85) {\tiny $R_2(x,y)$};
\node[above] at (0,-.5) {\tiny $t$};
\node[below] at (0,-2) {\tiny $k$};
\node[right] at (4,.6) {\tiny $w$};
\node[right] at (4,-1) {\tiny $z$};
\node[above] at (1,1.8) {\tiny $s$};
\node[below] at (-4,-1.6) {\tiny $R_2(x,y) \bar{*} R_1(x,y)$};
\node[left] at (3.25,-.55) {\tiny $R_2(w,z)$};
\node[above] at (3.1,-.2) {\tiny $R_1(w,z)$};
\end{tikzpicture}
\vspace{.2in}
		\caption{Diagram $K_4$.}
		\label{consumK22}
\end{figure}

Coloring the second singular knot $K_4$ by this singquandle gives the set of equations 
\[
\begin{cases}
                        x=k \bar{*} (R_2(x,y) \bar{*} R_1(x,y)) = 7k + 6x + 8y\\
                        y=(R_2(x,y) \bar{*} R_1(x,y))\bar{*} k = 4k+3x+5x^2+9y\\
                        t= R_1(x,y) \bar{*}y = 7x+4y\\
                        w= R_1(w,z) \bar{*} s= 4s+7w\\
                        z = s \bar{*} t=7s+4t\\
                        s= t \bar{*} R_1(w,z)=7t+4w\\
                        k = R_2(w,z) \bar{*}z = 5w+5w^2 +z.
\end{cases}
\]
These two systems give that the two singular knots $K_3$ and $K_4$ have 40 colorings by this singquandle.  We now compute the $2$-cocycle invariant using $\phi,\phi': X \times X \rightarrow \mathbb{Z}_4$ as defined above. The first singular knot $K_3$ has $\Phi_{\phi,\phi'}(K_3) 
%= \sum u^{\phi(x,y) + \phi'(x,y)}
= 10+10u+10u^2+10u^3$ and the second knot $K_4$ has $\Phi_{\phi,\phi'}(K_4) 
%= \sum u^{\phi(x,y) + \phi'(x,y)} 
= 10u+20u^2+10u^3$.
\end{example}

\begin{example}
In this example we use the $2$-cocycle invariant to distinguish the singular \emph{square} knot and \emph{granny} knot each with two singular crossings. Consider the oriented singquandle $(\mathbb{Z}_{10}, *, R_1, R_2)$ with cocycles $\phi(x,y) = 2(2+3x+3y)^3$ and $\phi'(x,y) = (1+x+y)^5$ from item (4) of Example \ref{ex1cocycles}.  Recall that $x*y = 3x -2y$, $x 
\bar{*}y = -3x+4y$, $R_1(x,y) =3x+8y$ and $R_2(x,y) = 4x+7y$.  

\begin{figure}[h]

\begin{tikzpicture}[use Hobby shortcut]
%diagram on the left
\begin{knot}[
consider self intersections=true,
%draft mode=crossings,
  ignore endpoint intersections=false,
  flip crossing/.list={4,7,9,12,6,8}
]
\strand ([closed]-2.5,2.5) ..(-1,0)[decoration={markings,mark=at position .7 with
    {\arrow[scale=3,>=stealth]{>}}},postaction={decorate}]..(-4,-1)..(-2,1)..(0,1).. (2,1)..(4,-1)..(1,0)..(2.5,2.5)..(3.5,1)..(2,-1)..(0,-1)..(-2,-1)..(-3.5,1);

\end{knot}
\node[circle,draw=black, fill=black, inner sep=0pt,minimum size=8pt] (a) at (-3.4,.5) {};
\node[circle,draw=black, fill=black, inner sep=0pt,minimum size=8pt] (a) at (3.4,.5) {};
\node[above] at (-2.5,1) {\tiny $R_1(x,y)$};
\node[above] at (-2.5,2.6) {\tiny $R_2(x,y)$};
\node[left] at (-2.5,-.4) {\tiny $x$};
\node[left] at (-4,-1) {\tiny $y$};
\node[right] at (3,1) {\tiny $w$};
\node[above] at (2,.5) {\tiny $z$};
\node[above] at (2,-.5) {\tiny $R_1(z,w)$};
\node[right] at (4,-.5) {\tiny $R_2(z,w)$};
\node[rotate=50] at (.5,2) {\tiny $R_1(x,y) *R_2(x,y)$};
\end{tikzpicture}
\vspace{.2in}
		\caption{A diagram of the singular \emph{square} knot.}
		\label{square}
\end{figure}
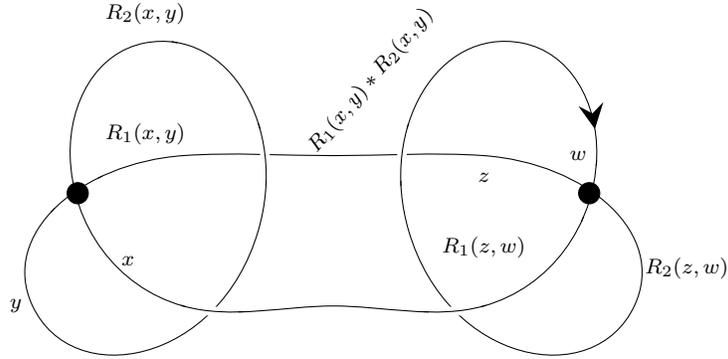

The coloring the singular square knot by the above singquandle gives the set of equations 
\[
\begin{cases}
                        x=R_1(z,w) =8w+3z\\
                        y=R_2(x,y) *x=y\\
                        z = (R_1(x,y)*R_2(x,y)) \bar{*} w =4w+7x\\
                        w=R_2(z,w) \bar{*} R_1(z,w)=w.
\end{cases}
\]
This system of equations has \emph{three} free variables giving that the number of colorings of the singular square knot is $10^3.$

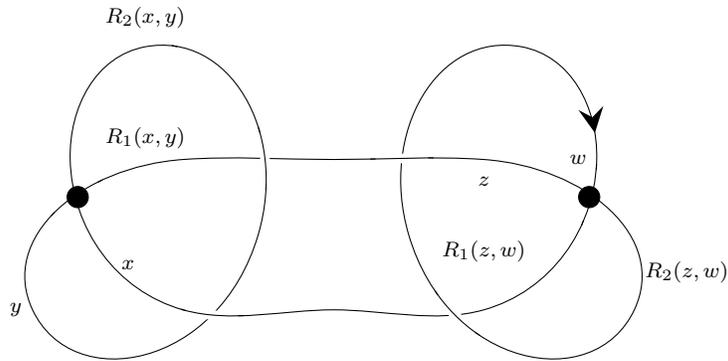
\begin{figure}[h]

\begin{tikzpicture}[use Hobby shortcut]
%diagram on the left
\begin{knot}[
consider self intersections=true,
%draft mode=crossings,
  ignore endpoint intersections=false,
  flip crossing/.list={4,7,9,12}
]
\strand ([closed]-2.5,2.5) ..(-1,0)[decoration={markings,mark=at position .7 with
    {\arrow[scale=3,>=stealth]{>}}},postaction={decorate}]..(-4,-1)..(-2,1)..(0,1).. (2,1)..(4,-1)..(1,0)..(2.5,2.5)..(3.5,1)..(2,-1)..(0,-1)..(-2,-1)..(-3.5,1);

\end{knot}
\node[circle,draw=black, fill=black, inner sep=0pt,minimum size=8pt] (a) at (-3.4,.5) {};
\node[circle,draw=black, fill=black, inner sep=0pt,minimum size=8pt] (a) at (3.4,.5) {};
\node[above] at (-2.5,1) {\tiny $R_1(x,y)$};
\node[above] at (-2.5,2.6) {\tiny $R_2(x,y)$};
\node[left] at (-2.5,-.4) {\tiny $x$};
\node[left] at (-4,-1) {\tiny $y$};
\node[right] at (3,1) {\tiny $w$};
\node[above] at (2,.5) {\tiny $z$};
\node[above] at (2,-.5) {\tiny $R_1(z,w)$};
\node[right] at (4,-.5) {\tiny $R_2(z,w)$};

\end{tikzpicture}
\vspace{.2in}
		\caption{A diagram of the singular \emph{granny} knot.}
		\label{granny}
\end{figure}

Coloring the singular granny knot by the above singquandle gives the set of equations 
\[
\begin{cases}
                        x=R_1(z,w) * R_2(z,w)= z\\
                        y=R_2(x,y) * x =y\\
                        z=R_1(x,y) * R_2(x,y)= x\\
                        w = R_2(z,w) * (R_1(x,y) *R_2(x,y))= w.
\end{cases}
\]
This system of equations has also \emph{three} free variables giving that the number of colorings of the singular granny knot is $10^3.$
  We now compute the $2$-cocycle invariant using $\phi, \phi': X \times X \rightarrow \mathbb{Z}_4$ as defined above. The singular \emph{square} knot has $\Phi_{\phi,\phi'}
  %= \sum u^{\phi(x,y) + \phi'(x,y)}
  = 378 + 250 u + 122 u^2 + 250 u^3$ and the singular \emph{granny} knot has $\Phi_{\phi,\phi'}
  %= \sum u^{\phi(x,y) + \phi'(x,y)} 
  = 370 + 250 u + 130 u^2 + 250 u^3$.  Thus this invariant distinguish them.
\end{example}

{\bf Acknowledgements}
The authors would like to thank Sam Nelson for helpful conversations regarding this article. Mohamed Elhamdadi was partially supported by Simons Foundation collaboration grant 712462.


\begin{thebibliography}{2}




\bib{BEHY}{article}{
   author={Bataineh, Khaled},
   author={Elhamdadi, Mohamed},
   author={Hajij, Mustafa},
   author={Youmans, William},
   title={Generating sets of Reidemeister moves of oriented singular links
   and quandles},
   journal={J. Knot Theory Ramifications},
   volume={27},
   date={2018},
   number={14},
   pages={1850064, 15},
   issn={0218-2165},
   review={\MR{3896310}},
   %doi={10.1142/S0218216518500645},
}


\bib{Birman}{article}{
   author={Birman, Joan S.},
   title={New points of view in knot theory},
   journal={Bull. Amer. Math. Soc. (N.S.)},
   volume={28},
   date={1993},
   number={2},
   pages={253--287},
   issn={0273-0979},
   review={\MR{1191478 (94b:57007)}},
   %doi={10.1090/S0273-0979-1993-00389-6},
}


\bib{CEGS}{article}{
   author={Carter, J. Scott},
   author={Elhamdadi, Mohamed},
   author={Gra\~{n}a, Matias},
   author={Saito, Masahico},
   title={Cocycle knot invariants from quandle modules and generalized
   quandle homology},
   journal={Osaka J. Math.},
   volume={42},
   date={2005},
   number={3},
   pages={499--541},
   issn={0030-6126},
   review={\MR{2166720}},
}

\bib{CENS}{article}{
   author={Carter, J. Scott},
   author={Elhamdadi, Mohamed},
   author={Nikiforou, Marina Appiou},
   author={Saito, Masahico},
   title={Extensions of quandles and cocycle knot invariants},
   journal={J. Knot Theory Ramifications},
   volume={12},
   date={2003},
   number={6},
   pages={725--738},
   issn={0218-2165},
   review={\MR{2008876}},
   %doi={10.1142/S0218216503002718},
}

\bib{CES}{article}{
   author={Carter, J. Scott},
   author={Elhamdadi, Mohamed},
   author={Saito, Masahico},
   title={Homology theory for the set-theoretic Yang-Baxter equation and
   knot invariants from generalizations of quandles},
   journal={Fund. Math.},
   volume={184},
   date={2004},
   pages={31--54},
   issn={0016-2736},
   review={\MR{2128041}},
   doi={10.4064/fm184-0-3},
}

\bib{CJKLS}{article}{
   author={Carter, J. Scott},
   author={Jelsovsky, Daniel},
   author={Kamada, Seiichi},
   author={Langford, Laurel},
   author={Saito, Masahico},
   title={Quandle cohomology and state-sum invariants of knotted curves and
   surfaces},
   journal={Trans. Amer. Math. Soc.},
   volume={355},
   date={2003},
   number={10},
   pages={3947--3989},
   issn={0002-9947},
   review={\MR{1990571}},
   %doi={10.1090/S0002-9947-03-03046-0},
}

\bib{CN}{article}{
   author={Ceniceros, Jose},
   author={Nelson, Sam},
   title={Virtual Yang-Baxter cocycle invariants},
   journal={Trans. Amer. Math. Soc.},
   volume={361},
   date={2009},
   number={10},
   pages={5263--5283},
   issn={0002-9947},
   review={\MR{2515811}},
   doi={10.1090/S0002-9947-09-04751-5},
}


\bib{CEHN}{article}{
author={Churchill, Indu R. U.},
	author={Elhamdadi, Mohamed},
	author={Hajij, Mustafa},
    author={Nelson, Sam},
	title={Singular Knots and Involutive Quandles},
	journal={ J. Knot Theory Ramifications 26 (2017), no. 14, 1750099, 14 pp.}
	volume={},
	date={},
	number={},
	pages={},
	issn={},
	review={\MR{1362791 (96i:57015)}},
	%doi={10.1016/0040-9383(94)00051-4},
}

\bib{EN}{book}{
   author={Elhamdadi, Mohamed},
   author={Nelson, Sam},
   title={Quandles---an introduction to the algebra of knots},
   series={Student Mathematical Library},
   volume={74},
   publisher={American Mathematical Society, Providence, RI},
   date={2015},
   pages={x+245},
   isbn={978-1-4704-2213-4},
   review={\MR{3379534}},
}

\bib{EMR}{article}{
   author={Elhamdadi, Mohamed},
   author={Macquarrie, Jennifer},
   author={Restrepo, Ricardo},
   title={Automorphism groups of quandles},
   journal={J. Algebra Appl.},
   volume={11},
   date={2012},
   number={1},
   pages={1250008, 9},
   issn={0219-4988},
   review={\MR{2900878}},
   %doi={10.1142/S0219498812500089},
}

\bib{Fiedler}{article}{
	author={Fiedler, Thomas},
	title={The Jones and Alexander polynomials for singular links},
	journal={J. Knot Theory Ramifications},
	volume={19},
	date={2010},
	number={7},
	pages={859--866},
	issn={0218-2165},
	review={\MR{2673687 (2012b:57024)}},
	%doi={10.1142/S0218216510008236},
}
	
	\bib{Gemein}{article}{
	author={Gemein, Bernd},
	title={Singular braids and Markov's theorem},
	journal={J. Knot Theory Ramifications},
	volume={6},
	date={1997},
	number={4},
	pages={441--454},
	issn={0218-2165},
	review={\MR{1466593 (98i:57008)}},
	%doi={10.1142/S0218216597000297},
}

\bib{Joyce}{article}{
   author={Joyce, David},
   title={A classifying invariant of knots, the knot quandle},
   journal={J. Pure Appl. Algebra},
   volume={23},
   date={1982},
   number={1},
   pages={37--65},
   issn={0022-4049},
   review={\MR{638121}},
   %doi={10.1016/0022-4049(82)90077-9},
}

\bib{JL}{article}{ 
title={An invariant for singular knots},
  author={Juyumaya, Jes{\'u}s},
   author={  Lambropoulou, Sofia},
  journal={Journal of Knot Theory and Its Ramifications},
  volume={18},
  number={06},
  pages={825--840},
  year={2009},
  publisher={World Scientific}
}

\bib{Matveev}{article}{
   author={Matveev, S. V.},
   title={Distributive groupoids in knot theory},
   language={Russian},
   journal={Mat. Sb. (N.S.)},
   volume={119(161)},
   date={1982},
   number={1},
   pages={78--88, 160},
   issn={0368-8666},
   review={\MR{672410}},
}
\bib{NOS}{article}{
   author={Nelson, Sam},
   author={Oyamaguchi, Natsumi},
   author={Sazdanovic, Radmila},
   title={Psyquandles, singular knots and pseudoknots},
   journal={Tokyo J. Math.},
   volume={42},
   date={2019},
   number={2},
   pages={405--429},
   issn={0387-3870},
   review={\MR{4106586}},
}

	
\bib{Oyamaguchi}{article}{
   author={Oyamaguchi, Natsumi},
   title={Enumeration of spatial 2-bouquet graphs up to flat vertex isotopy},
   journal={Topology Appl.},
   volume={196},
   date={2015},
   number={part B},
   part={part B},
   pages={805--814},
   issn={0166-8641},
   review={\MR{3431017}},
   %doi={10.1016/j.topol.2015.05.049},
}


\bib{Paris}{article}{
	author={Paris, Luis},
	title={The proof of Birman's conjecture on singular braid monoids},
	journal={Geom. Topol.},
	volume={8},
	date={2004},
	pages={1281--1300 (electronic)},
	issn={1465-3060},
	review={\MR{2087084}},
	%doi={10.2140/gt.2004.8.1281},
}

\bib{V}{article}{
   author={Vassiliev, V. A.},
   title={Cohomology of knot spaces},
   conference={
      title={Theory of singularities and its applications},
   },
   book={
      series={Adv. Soviet Math.},
      volume={1},
      publisher={Amer. Math. Soc., Providence, RI},
   },
   date={1990},
   pages={23--69},
   review={\MR{1089670}},
}

\end{thebibliography}
\end{document}